\documentclass[10pt,leqno]{amsart}
\usepackage{amssymb,amsthm}
\usepackage{amsmath}
\usepackage{setspace}
\usepackage{dcpic,pictexwd}
\usepackage[all]{xy}
\usepackage[pdftex]{graphicx}
\usepackage{mathrsfs}
\usepackage[pdftex]{hyperref}
\usepackage{bbm}

\theoremstyle{definition}
 \newtheorem{definition}{Definition}[section]

\theoremstyle{plain}

\theoremstyle{plain}
 \newtheorem{theorem}[definition]{Theorem}
 
\theoremstyle{definition}

\theoremstyle{plain}

\theoremstyle{plain}
 \newtheorem{corollary}[definition]{Corollary}

\theoremstyle{remark}
 \newtheorem{remark}[definition]{Remark}

\theoremstyle{definition}

\theoremstyle{plain}

\newcommand{\Ext}{\mathrm{Ext}}
\newcommand{\End}{\mathrm{End}}

\newcommand{\Hom}{\mathrm{Hom}}

\newcommand{\Ca}{\mathcal{C}}
\newcommand{\Fun}{\mathrm{F}}

\newcommand{\Def}{\mathrm{Def}}
\newcommand{\Sets}{\mathrm{Sets}}

\newcommand{\SEnd}{\underline{\End}}
\newcommand{\A}{\Lambda}

\newcommand{\m}{\mathfrak{m}}
\renewcommand{\k}{\Bbbk}

\newcommand{\invlim}{\varprojlim}

\hypersetup{
    bookmarks=true,         
    unicode=false,          
    pdftoolbar=true,        
    pdfmenubar=true,        
    pdffitwindow=false,     
    pdfstartview={FitH},    
    pdfnewwindow=true,      
    colorlinks=true,       
    linkcolor=red,          
    citecolor=blue,        
    filecolor=magenta,      
    urlcolor=blue           
}

\title[A note on deformations of finite dimensional modules]{A note on deformations of finite dimensional modules over $\k$-algebras} 
\thanks{}
\author[Rizzo]{Pedro Rizzo}
\address{Instituto de Matem\'aticas, Universidad de Antioquia, Medell{\'\i}n, Antioquia, Colombia}
\email{pedro.hernandez@udea.edu.co}
\author[V\'elez-Marulanda]{Jos\'e A. V\'elez-Marulanda}
\address{Department of Mathematics, Valdosta State University, Valdosta, GA,  United States of America}
\email{javelezmarulanda@valdosta.edu (Corresponding author)}

\begin{document}
\renewcommand{\labelenumi}{\textup{(\roman{enumi})}}
\renewcommand{\labelenumii}{\textup{(\roman{enumi}.\alph{enumii})}}
\numberwithin{equation}{section}

\begin{abstract}

Let $\k$ be a field, and let $\A$ be a (not necessarily finite dimensional) $\k$-algebra. Let $V$ be a left $\A$-module such that is finite dimensional over $\k$. Assume further that $V$ has a weak universal deformation ring $R^w(\A,V)$, which is a complete Noetherian commutative local $\k$-algebra with residue field $\k$.  We prove in this note that under certain conditions on the $\A$-module $V$, that if $R^w(\A,V)$ is a quotient of $\k[\![t]\!]$, then $R^w(\A,V)$ is either isomorphic to $\k$, or $\k[\![t]\!]$, or to $\k[\![t]\!]/(t^N)$ for some integer $N\geq 2$.  
\end{abstract}
\maketitle

\section{Introduction}

Assume that $\A$ is a finite dimensional self-injective $\k$-algebra.  If $V$ has stable endomorphism ring is isomorphic to $\k$, then it follows from the results in \cite{blehervelez} that the weak deformation functor $\widehat{\Fun}_V^w(-)$ is naturally equivalent to the deformation functor $\widehat{\Fun}_V(-)$  and that $V$ has a universal deformation ring $R(\A,V)$. Consequently, $V$ has also a weak universal deformation ring $R^w(\A,V)$ that is isomorphic to $R(\A,V)$. 
Moreover, in \cite{fonce-giraldo-rizzo-velez}, the authors obtain the same result for the case when $\A= \widehat{\Gamma}$, which is the repetitive algebra of a finite-dimensional $\k$-algebra $\Gamma$ (as defined in \cite{wald}). It should be noted that in this particular situation, $\A$ is an infinite-dimensional $\k$-algebra. Furthermore, in \cite{blehervelez}, \cite{calderon-giraldo-rueda-velez}, \cite{fonce-giraldo-rizzo-velez}, and \cite{velez}, several combinatorial methods from the representation theory of algebras were employed to compute universal deformation rings of modules whose stable endomorphism ring is isomorphic to $\k$. It is worth mentioning that many of these calculations utilize similar methods or involve adjustments from one situation to the other. For instance, we refer to \cite[Proof of Claim 4.5]{calderon-giraldo-rueda-velez} and \cite[Proof of Theorem 1.2]{fonce-giraldo-rizzo-velez} for examples of such similarities and adaptations.

In this note, we present a general criterion and a unified technique for computing universal deformation rings under broad assumptions, encompassing the aforementioned results as special cases. Moreover, these calculations can be applied to other scenarios as well. Our objective in this note is to establish the following result.

\begin{theorem}\label{thm1}
Let $\A$ be a (not necessarily finite dimensional) $\k$-algebra and $V$ be a left $\A$-module with $\dim_\k V<\infty$. Assume that $V$ has a weak universal deformation ring $R^w(\A,V)$ which is a quotient of the ring of formal series $\k[\![t]\!]$, and that there exists an ordered sequence of finite dimensional left $\A$-modules $\mathscr{L}_V= \{V_0,V_1,\ldots\}$ with $V_0 = V$ and such that for each $\ell \geq 1$, there exist a surjective morphism  $\epsilon_\ell: V_\ell\to V_{\ell-1}$, and an injective morphism $\iota_\ell: V_{\ell-1}\to V_\ell$ such that the composition $\sigma_\ell = \iota_\ell\circ \epsilon_\ell$ satisfies that $\ker \sigma_\ell = V_0$, $\mathrm{im}\,\sigma_\ell^\ell \cong V_0$, $\sigma_\ell^{\ell+1}=0$, and $\mathscr{L}_V$ is maximal with respect to these properties. 
\begin{enumerate}
\item If $\mathscr{L}_V$ is finite, and its last element, say $V_N$, satisfies  $\dim_\k\Hom_\A(V_N,V) =1$ and $\Ext_\A^1(V_N,V)=0$, then  $R^w(\A,V)\cong \k[\![t]\!]/(t^{N+1})$.
\item If $\mathscr{L}_V$ is infinite, then $R^w(\A,V)\cong \k[\![t]\!]$.
\end{enumerate}   
\end{theorem}

\section{Preliminaries}

As before, we assume that $\k$ is a fixed field of arbitrary characteristic. We denote by $\widehat{\Ca}$ the category of all complete local commutative Noetherian $\k$-algebras with residue field $\k$. In particular, the morphisms in $\widehat{\Ca}$ are continuous $\k$-algebra homomorphisms that induce the identity map on $\k$.  Let $\A$ be a fixed and not necessarily finite dimensional  $\k$-algebra, and let $R$ be a fixed but arbitrary object in $\widehat{\Ca}$.  We denote by $R\A$ the tensor product of $\k$-algebras $R\otimes_\k\A$. Let $V$ be a fixed left $\A$-module with $\dim_\k V <\infty$. A {\it weak lift} of $V$ over an object $R$ in $\widehat{\Ca}$ is a finitely generated left $R\A$-module $M$ that is also free over $R$ such that $\k\otimes_R M\cong V$ as left $\A$-modules. A {\it weak deformation} of $V$ over $R$ is a isomorphism class $[M]$ of weak lifts of $V$ over $R$. We denote by $\Def^w_\A(V,R)$ the set of all weak deformations of $V$ over $R$. The {\it weak deformation functor} of $V$ is the covariant functor $\widehat{\Fun}_V^w: \widehat{\Ca}\to \Sets$ that sends every object $R$ in $\widehat{\Ca}$ to $\Def^w_\A(V,R)$, and which sends any morphism $\theta: R\to R'$ in $\widehat{\Ca}$ to the morphism $\widehat{\Fun}_V^w(\theta): \Def_\A^w(V,R)\to \Def_\A^w(V,R')$ that is defined as $\widehat{\Fun}_V^w([M])=[R'\otimes_{R,\theta}M]$. If $R$ is the ring of dual numbers $\k[\epsilon]$ with $\epsilon^2=0$, then $t_V=\Fun_V(\k[\epsilon])$ is called the {\it tangent space} of $\Fun_V$. 

Assume that there exists an object $R^w(\A,V)$ in $\widehat{\Ca}$ that represents $\widehat{\Fun}_V^w(-)$ in the sense that there is a natural equivalence between the functors $\widehat{\Fun}_V^w(-)$ and $\Hom_{\widehat{\Ca}}(R^w(\A,V), -)$. In this situation, we call $R^w(\A,V)$ the {\it weak universal deformation ring} of $V$.  Assume further that there exists an isomorphism of $\k$-vector spaces $t_V\cong \Ext_\A^1(V,V)$. If $\dim_\k t_V=n$, then it follows that $R^w(\A,V)$ is a quotient of the ring of formal power series $\k[\![t_1,\ldots,t_n]\!]$. In this note, we are interested in finite dimensional $\A$-modules $V$ that have a weak deformation ring $R^w(\A,V)$ that is a quotient of $\k[\![t]\!]$.

\begin{remark}\label{remexam}
In order to prove the main result in this note, we need the following definition and property of morphisms between objects in $\widehat{\Ca}$. Following \cite{sch}, if $R$ is an object in $\widehat{\Ca}$, we denote by $t^\ast_R$ the quotient $\m_R/\m_R^2$ and call it the  {\it Zariski cotangent space} of $R$ over $\k$. Let $\theta: R\to R'$ be a morphism in $\widehat{\Ca}$.  It follows by \cite[Lemma 1.1]{sch} that $\theta$ is surjective if and only if the induced map of cotangent spaces $\theta^\ast: t^\ast_R\to t^\ast_{R'}$ is surjective. 
\end{remark}

\section{Proof of main result}
\begin{proof}[Proof of Theorem \ref{thm1}] Let $\ell \geq 1$ be a fixed integer and assume that $V_\ell$ and $\sigma_\ell$ are as in the hypothesis of Theorem \ref{thm1}. It follows that the $\A$-module  $V_\ell$ is naturally a $\k[\![t]\!]/(t^{\ell+1})\otimes_\k\A$-module by letting $t$ act on $x\in V_\ell$ as $t\cdot x = \sigma_\ell(x)$. In particular, $tV_\ell\cong V_{\ell-1}$. Assume that $d=\dim_\k V$ and let $\{\bar{r}_1,\ldots,\bar{r}_d\}$ be a fixed basis of $V$ over $\k$. By using the isomorphism $V_\ell/tV_\ell\cong V_0$, we can lift the elements $\bar{r}_1,\ldots,\bar{r}_d$ to corresponding elements $r_1,\ldots,r_d\in V_\ell$ that are linearly independent over $\k$ and such that $\{t^sr_1,\ldots,t^sr_d: 1\leq s\leq \ell\}$ is a $\k$-basis of $tV_\ell\cong V_{\ell-1}$, which implies that $\{r_1,\ldots,r_d\}$ is  a $\k[\![t]\!]/(t^{\ell+1})$-basis of $V_\ell$, i.e. $V_\ell$ is free over $\k[\![t]\!]/(t^{\ell+1})$. Note also that $V_\ell$ lies in a short exact sequence of $\A$-modules
\begin{equation*}
0\to tV_\ell\to V_\ell\to \k\otimes_{\k[\![t]\!]/(t^{\ell+1})}V_\ell\to 0,
\end{equation*}
which implies that there exists an isomorphism of $\A$-modules $\phi_\ell:\k\otimes_{\k[\![t]\!]/(t^{\ell+1})}V_\ell\to V_0$, which implies that $V_\ell$ induces a weak lift of $V_0$ over $\k[\![t]\!]/(t^{\ell+1})$. 

In order to prove (i), assume next that $\mathscr{L}_V$ is finite and that the last term of this sequence, say $V_N$, satisfies that $\dim_\k\Hom_\A(V_N,V_0)=1$ and $\Ext_\A^1(V_N,V_0)=0$. Consider the weak lift $V_N$ of $V_0$ over $\k[\![t]\!]/(t^{N+1})$. Since $\widehat{\Fun}_V^w(-)$ is representable, it follows that there exists a unique morphism $\theta: R^w(\A, V_0)\to \k[\![t]\!]/(t^{N+1})$ in $\widehat{\Ca}$ such that 
\begin{equation*}
V_N\cong \k[\![t]\!]/(t^{N+1})\otimes_{R^w(\A,V_0), \theta}U(\A,V_0),
\end{equation*}
where $U(\A,V_0)$ is the weak lift corresponding to the weak deformation of $V_0$ over $R^w(\A,V_0)$. On the other hand, since $V_1$ is a weak lift of $V_0$ over $\k[\![t]\!]/(t^2)$, there exists a unique morphism $\theta':R^w(\A,V_0)\to \k[\![t]\!]/(t^2)$ in $\widehat{\Ca}$ such that  
\begin{equation*}
V_1\cong \k[\![t]\!]/(t^2)\otimes_{R^w(\A,V_0),\theta'}U(\A,V_0).
\end{equation*}
By considering the natural projection $\pi_{N+1,2}: \k[\![t]\!]/(t^{N+1})\to \k[\![t]\!]/(t^2)$ and the weak lift $(U',\phi_{U'})$  of $V_0$ over $\k[\![t]\!]/(t^2)$ corresponding to the morphism $\pi_{N+1,2}\circ \theta$, we obtain 
\begin{align*}
U'&\cong \k[\![t]\!]/(t^2)\otimes_{R(\A_N,V_0),\pi_{N+1,2}\circ \theta}U(\A_N,V_0)\\
&\cong \k[\![t]\!]/(t^2)\otimes_{\k[\![t]\!]/(t^{N+1}),\pi_{N+1,2}}\left(\k[\![t]\!]/(t^N)\otimes_{R(\A_N,V_0),\theta}U(\A_N,V_0)\right)\\
&\cong \k[\![t]\!]/(t^2)\otimes_{\k[\![t]\!]/(t^{N+1}),\pi_{{N+1},2}}V_N\\
&\cong V_N/t^2V_N\cong V_1. 
\end{align*} 
The uniqueness of $\theta'$ implies that $\theta'= \pi_{N+1,2}\circ \theta$, and since $\theta'$ is surjective, it follows that $\theta$ is also surjective. We claim that $\theta$ is an isomorphism. If this is false, then there exists a surjective $\k$-algebra homomorphism $\theta_0: R^w(\A,V_0)\to \k[\![t]\!]/(t^{N+2})$ in $\widehat{\Ca}$ such that $\pi_{N+2,N+1}\circ \theta$, where $\pi_{N+2,N+1}: \k[\![t]\!]/(t^{N+2})\to \k[\![t]\!]/(t^{N+1})$ is the natural projection. Let $M_0$ be a weak lift of $V_0$ over  $\k[\![t]\!]/(t^{N+2})$ corresponding  to $\theta_0$. Since the kernel of $\pi_{N+2,N+1}$ is $(t^{N+1})/(t^{N+2})$, it follows that $M_0/t^{N+1}M_0\cong V_N$. Consider the $\k[\![t]\!]/(t^{N+2})\otimes_\k\A$-module homomorphism $g:M_0\to t^{N+1}M_0$ defined as $g(x)=t^{N+1}x$ for all $x\in M_0$. Since $M_0$ is free over $\k[\![t]\!]/(t^{N+2})$ it follows that the kernel of $g$ is isomorphic to $tM_0$. Thus, $M_0/tM_0\cong t^{N+1}M_0$ for $g$ is a surjection. Therefore $t^{N+1}M_0\cong V_0$, and thus we obtain a short exact sequence of $\k[\![t]\!]/(t^{N+2})\otimes_\k\A$-modules 
\begin{equation}\label{eqn1}
0\to V_0\to M_0\to V_N\to 0.
\end{equation}

Since by assumption we have $\Ext_\A^1(V_N,V_0)=0$, it follows that the sequence (\ref{eqn1}) splits as a sequence of $\A$-modules. Hence, $M_0=V_0\oplus V_N$ as $\A$-modules. Identifying the elements of $M_0$ as $(v,x)$ with $v\in V_0$ and $x\in V_N$, we see that $t$ acts on $(v,x)\in M_0$ as $t\cdot(v,x)=(\mu(x),\sigma_N(x))$, where $\mu :V_N\to V_0$ is a surjective $\A$-module homomorphism. Since the surjection $\sigma_N^N: V_N\to V_0$ is non-zero and since by hypothesis $\dim_\k\Hom_\A(V_N,V_0)=1$, it follows that there exists $c\in \k^\ast$ such that $\mu = c\sigma_N^N$, which implies that the kernel of $\mu$ is $tV_N$. Therefore $t^{N+1}(v,x)=(\mu(t^Nx),\sigma_N^{N+1}(x))=(0,0)$ for all $v\in V_0$ and $x\in V_N$. This contradicts that $t^{N+1}M_0\cong V_0$. Thus $\theta: R^w(\A,V_0)\to \k[\![t]\!]/(t^{N+1})$ is an isomorphism. This proves (i).  

Next assume that $\mathscr{L}_V$ is infinite, and let $W=\invlim_\ell V_\ell$. By letting $t$ act on $W$ as $\invlim\epsilon_\ell$, we obtain that $W$ is a $\k[\![t]\!]\otimes_\k\A$-module and $\k\otimes_RW\cong W/tW$, and by arguing as before, it follows that $W$ is a weak lift of $V_0$ over $\k[\![t]\!]$. Therefore, there exists a unique $\k$-algebra homomorphism $\theta: R^w(\A,V_0)\to  \k[\![t]\!]$ in $\widehat{\Ca}$ which corresponds to the weak deformation induced by $W$. Note that since $W/t^2W\cong V_1$ as $\A$-modules, we obtain that $W/t^2W$ defines a non-trivial lift of $V_0$ over $\k[\![t]\!]/(t^2)$ and thus there exists a unique morphism $\theta':R^w(\A,V_0)\to \k[\![t]\!]/(t^2)$. Note that since the cotangent space (as Remark \ref{remexam}) of $\k[\![t]\!]/(t^2)$ is $1$-dimensional over $\k$, it follows that $\theta'$ is also surjective. Moreover, if $\theta'':\k[\![t]\!]\to \k[\![t]\!]/(t^2)$ is the canonical projection, it follows by the uniqueness of $\theta'$ that $\theta'=\theta''\circ \theta$. Thus again by Remark \ref{remexam}, since $(\theta'')^\ast$ is an isomorphism, we obtain that $\theta^\ast$ and thus $\theta$ is surjective. Therefore by using that  $R^w(\A,V_0)$ is a quotient of $\k[\![t]\!]$, we conclude that $\theta$ is an isomorphism. This proves (ii) and finishes the proof of Theorem \ref{thm1}. 
\end{proof}

\begin{corollary}\label{cor1}
Under the situation of Theorem \ref{thm1} (i), if $N = 0$, then $\Ext_\A^1(V,V)=0$ and $R^w(\A,V)\cong \k$.
\end{corollary}

\begin{proof}
Note that if $N=0$, then $\mathscr{L}_V=\{V=V_0\}$. Assume that $\Ext_\A^1(V,V)\not=0$, i.e. there exists a non-splitting short exact sequence of $\A$-modules 
\begin{equation}
0\to V_0\xrightarrow{\iota_1}V_1\xrightarrow{\epsilon_1}V\to 0.
\end{equation}
Note that the morphism $\sigma_1=\iota_1\circ \epsilon_1$ satisfies $\sigma_1^2 =0$, $\ker \sigma_1\cong V_0$ and $\mathrm{im}\,\sigma_1\cong V_0$. Thus $\{V_0,V_1\}\subset \mathscr{L}_V$ which is a contradiction. Thus $\Ext_\A^1(V,V)=0$. The second affirmation of Corollary \ref{cor1} follows directly from Theorem \ref{thm1} (i).  
\end{proof}

\begin{remark}
\begin{enumerate}
\item Assume that $\A$ is either a finite dimensional self-injective $\k$-algebra or the repetitive algebra of a finite dimensional $\k$-algebra. If $V$ is a finite dimensional module such that $\SEnd_\A(V)=\k$, then $V$ has a universal deformation ring $R(\A,V)$ (in the sense of \cite{blehervelez}) and moreover, $V$ also has a weak universal deformation ring $R^w(\A,V)$ with $R^w(\A,V)\cong R(\A,V)$.  Therefore, Theorem \ref{thm1} recovers the results in e.g. \cite[Claims 4.4 \& 4.5]{calderon-giraldo-rueda-velez}, \cite[Theorem 1.2 (ii), (iii)]{fonce-giraldo-rizzo-velez}), \cite[Claims 4.1.1 \& 4.1.2]{velez}, which involve symmetric (thus self-injective) special biserial $\k$-algebras (in the sense of \cite{wald}).
\item The condition concerning the weak universal deformation ring $R^w(\A,V)$ being a quotient of the ring of formal series $\k[\![t]\!]$ is commonly encountered, as demonstrated by the aforementioned results. In fact, a more commonly encountered condition, which implies this, is when the dimension of the tangent space $t_V \cong \Ext_\A^1(V,V)$ is equal to one as a $\k$-vector space.
\end{enumerate}
\end{remark}

\bibliographystyle{amsplain}
\bibliography{Deformation_modules_algebras}   
\end{document}